\documentclass[12pt]{amsart}
\usepackage{amsmath,amssymb}
\usepackage{color}
\theoremstyle{plain}
\newtheorem{thm}{Theorem}[section]
\newtheorem{lem}[thm]{Lemma}

\newtheorem{cor}[thm]{Corollary}
\newtheorem{pro}[thm]{Proposition}
\newtheorem{ex}[thm]{Example}
\newtheorem{que}[thm]{Question}
\newtheorem{prob}[thm]{Problem}

\newtheorem{df}[thm]{Definition}
\newtheorem{rem}[thm]{Remark}

\begin{document}
\title{Discrete homogeneity and ends of manifolds}

\author{Vitalij A.~Chatyrko and Alexandre Karassev}

\begin{abstract} It is shown that a connected non-compact metrizable manifold of dimension $\ge 2$ is strongly discrete homogeneous if and only if it has one end (in the sense of Freudenthal compactification).
\end{abstract}

\makeatletter
\@namedef{subjclassname@2020}{\textup{2020} Mathematics Subject Classification}
\makeatother

\keywords{discrete homogeneity, manifold, ends}

\subjclass[2020]{Primary 57S05; Secondary 54F15, 54D35}

\maketitle

\section{Introduction} All spaces are assumed to be metrizable, and all maps continuous.
Recall  that a topological space $X$ is called {\it $n$-homogeneous} for a positive integer $n$ if
for any two  subsets $A$ and $B$ of $X$ with $|A|=|B|=n$  there exists a homeomorphism $f\colon X \to X$, such that $f(A)=B$ \cite{Bu}. Here, for a set $S$ by $|S|$ we denote its cardinality. Further, $X$ is called {\it strongly $n$-homogeneous} if
any bijection  $f\colon A \to B$ between  two such subsets extends to a homeomorphism  $f\colon X \to X$.

 Recall that a family of subsets $\mathcal D$ of a topological space $X$ is called {\it discrete} if each point $x \in X$ has a neighbourhood  $Ox$ which  intersects at most one  set from $\mathcal D$. If all elements of $\mathcal D$ are singletons, we obtain a definition of a {\it discrete subset} of $X$. Note that this definition implies that all discrete subsets of $X$ are closed in $X$.

The following stronger properties of topological spaces were considered in \cite{ChK}.
\begin{df}\label{main_def}  A space $X$ is called
\begin{itemize}
\item discrete homogeneous (br. DH) if for any two discrete subsets $A$ and $B$ of $X$ with $|A|=|B|$
 there exists a homeomorphism
$f\colon X\to X$ such that $f(A)= B$;

\item  strongly discrete homogeneous (br. sDH) if for any two discrete subsets $A$ and $B$ of $X$ and any bijection $f\colon A\to B$ this bijection extends to a homeomorphism of $X$ onto itself.
\end{itemize}
\end{df}

Note that a compact space is (strongly) discrete homogeneous iff it is (strongly) $n$-homogeneous for any $n.$ However, any infinite discrete space $X$ is strongly $n$-homogeneous for every positive integer $n$, but not discrete homogeneous.

We say that a homeomorphism $f\colon X\to X$ is \textit{supported} on a set $U$ if $f$ is the identity on $X\setminus U.$ Recall also that a space $X$ is called {\it strongly locally homogeneous} (br. SLH) if it has a basis of open sets $\mathcal B$ such that for every $U\in\mathcal B$ and any two points $x$ and $y$ in $U$ there exists a homeomorphism $f\colon X\to X$ supported on $U$ and such that $f(x) =y$ \cite{F}.

By an {\it $m$-manifold} we mean a space locally homeomorphic to $\mathbb R^m$. Note that any manifold is strongly locally homogeneous.

\begin{rem}\label{slh} {\rm It was shown in \cite{Ba} that  a connected strongly locally homogeneous space, no two-point subset of which separates it, is strongly $n$-homogeneous for any $n$. Therefore any compact connected manifold is strongly $n$-homogeneous for any $n$, and hence strongly discrete homogeneous.}
\end{rem}

It was shown in \cite[Theorem~3.3]{P} that $\mathbb R^n$ is strongly discrete homogeneous and that in fact the homeomorphism mapping one discrete subset of $\mathbb R^n$ to another can be chosen to be isotopic to identity.
In \cite{ChK} we independently obtained the first part of this result.  However, it was also observed in \cite{ChK} that the $2$-manifold $\mathbb R\times S^1$, where  $S^1$ is the circle, is not discrete homogeneous, and the following question was posed.

 \begin{que} Which metrizable manifolds are DH or sDH?
 \end{que}

 In this paper we answer the above question.

\section{Ends of spaces}

\begin{df}\label{general_cont} A space is called a generalized continuum if it is locally compact, $\sigma$-compact, connected, and locally connected space.
\end{df}

Note that in the literature one can find a slightly different definition of generalized continuum. Namely, the local connectedness condition is often omitted.

\begin{ex} Any connected manifold is a generalized continuum.	
\end{ex}

Let $X$ be a generalized continuum as in Definition~\ref{general_cont}.
There exists a sequence (generally speaking, not unique) of compact subsets $\{ X_n\}_{n=1}^\infty$ of $X$ such that $X_n\subset {\rm int} X_{n+1}$ for all $n=1,2, \dots$  and $\cup_{n=1}^\infty X_n = X$.

\begin{df}
By an end of $X$ with respect to $\{ X_n\}_{n=1}^\infty$ we mean a sequence of subsets $\{ Y_n\}_{n=1}^\infty$ such that for all $n$  the set $Y_n$ is a non-empty connected component of $X\setminus X_n$, and $Y_{n+1}\subset Y_n$.
\end{df}
It is easy to see that if a generalized continuum is compact than it has no ends.

We call a  connected component $C$ of $X \setminus X_n$  {\it bounded} if there exists $m$ such that  $C \subset X_m.$ Otherwise, $C$ is called {\it unbounded}.

\begin{rem}\label{remark_components} {\rm Below we summarize some facts about connected components of $X \setminus X_n,$  $n = 1,2, \dots $
\begin{itemize}
\item[(a)] The number of connected components $C$ of $X\setminus X_n$ such that $C \cap(X \setminus X_{n+1})\ne \varnothing$
is finite.
\item[(b)] If $X$ is non-compact, there exists at least one unbounded component of $X \setminus X_n$.

\item[(c)] If $X$ is non-compact, each connected component of $X \setminus X_{n+1}$  is contained in some connected component of $X \setminus X_n$, and each unbounded component  of $X \setminus X_n$ contains at least one unbounded component of $X \setminus X_{n+1}$. This implies that any non-compact generalized continuum has at least one end.
\end{itemize}}
\end{rem}

\begin{rem}
\begin{itemize}
\item[(a)] A generalized continuum has no ends iff it is compact.
\item[(b)] A generalized continuum has at most one end iff each of its compact subsets is contained in a compact subset with connected complement.
\end{itemize}
\end{rem}

The collection of all ends of $X$ is denoted by $E(X)$. Let $F(X) =X\cup E(X)$ be the disjoint union of $X$ and the set of its ends. For each open set $U$ let
$$E_U = \{ E=\{ Y_n\}_{n=1}^\infty \in E(X) \colon Y_n\subset U\mbox{ for some }n\}.$$
 One can introduce a base of topology on $F(X)$ consisting of sets of the form $E_U\cup U$ for all open subsets $U$ of $X.$ It turns out that $F(X)$, endowed with this topology, is a compactification of $X$, called the Freudenthal compactification \cite{F}.  Moreover, $F(X)$ and $E(X)$ do not depend on the choice of the sequence $\{ X_n\}_{n=1}^\infty.$ Note that if $|E(X)| = 1$ then $F(X)$ is the Alexandrov's one-point compactification of $X$.

The following fact is well-known. We include its proof for completeness.

\begin{pro}\label{extension} Let $X$ be a generalized continuum and $f\colon X\to X$ a homeomorphism. There exists a homeomorphism $\overline{f}\colon F(X)\to F(X),$ extending $f.$
\end{pro}

\begin{proof} For all $x\in X$ put $\overline{f}(x) = f(x).$ Let $\{ X_n\}_{n=1}^\infty$  be a sequence of compact subsets $\{ X_n\}_{n=1}^\infty$ of $X$ such that $X_n\subset {\rm int} X_{n+1}$ for all $n=1,2, \dots$  and $\cup_{n=1}^\infty X_n = X$. Then $\{ f(X_n)\}_{n=1}^\infty$ is a sequence with the same properties. Further, if $E=\{ Y_n\}_{n=1}^\infty$ is an end of $X$ with respect to $\{ X_n\}_{n=1}^\infty$ then $\{ f(Y_n)\}_{n=1}^\infty$ is an end of $X$ with respect to
 $\{f(X_n)\}_{n=1}^\infty.$ We let $\overline{f}(E) = \{f(Y_n)\}_{n=1}^\infty.$ It is straightforward to check that $\overline{f}$ is an autohomeomorphism of $F(X).$
\end{proof}

\section{Preliminary lemmas}

Recall that a map between topological spaces is called {\it proper} if preimage of any compact subspace is compact.

\begin{lem}\label{generalposition} Let $X$ be an $m$-manifold, $m\ge 3, $ and let $Z$ be a topological sum of countable collection of unit intervals. For any proper map $f\colon Z\to X$ and for any function $\epsilon\colon Z\to (0,1)$ there exists an embedding  $g\colon Z\to X$ such that $\textrm{dist} (f(z),g(z))<\epsilon (z)$ for all $z$ in $Z.$
\end{lem}
\begin{proof} It is a well-known fact that every $3$-manifold can be triangulated. Therefore, the statement of the lemma for $m=3$
	follows from the classical general position theorem in PL-category, see e.g. \cite[Corollary 4.4]{Br}. For $m>3$ the statement follows from \cite[Topological general position lemma 1]{D} (note that the manifolds considered in \cite{D} are topological manifolds with boundary).
\end{proof}

For the proof of the following see e.g. \cite[p. 305]{E}.

\begin{lem}\label{enlargement} Let $\mathcal D$ be a discrete family in a space $X$. For each $D\in \mathcal D$ there exists an open set $U_{D}$ containing $D$ such that the family $\{ \overline{U}_D\colon D\in \mathcal D\}$ is discrete.
\end{lem}

Recall that an {\it isotopy} between two homeomorphisms $f\colon X\to X$ and $g\colon X\to X$ is a map $H\colon X\times [0,1]\to X$ such that $f(x) =H(x,0)$ and $g(x)= H(x,1)$ for all $x$ in $X$, and the map $h_t\colon X\to X$ defined by $h(x) = H(x,t)$ is a homeomorphism for all $t\in[0,1].$ We also say that $f$ is isotopic to $g$ via $H$. Further, we say that an isotopy $H$ is \textit{supported on a subset $Y$ of $X$} if $h_t$ is supported on $Y$ for all $t.$

 \begin{df} A space $X$ is {\it isotopically strongly locally homogeneous} (br. isotopically SLH) if it has a basis of open sets $\mathcal B$ such that for every $U\in\mathcal B$ and any two points $x$ and $y$ in $U$ there exists a homeomorphism $f\colon X\to X$, supported on $U$, and isotopic to the identity via an isotopy supported on $U$,  such that $f(x) =y.$
 \end{df}

\noindent Note that  the above definition implies that the sets $U\in \mathcal B$ are connected.

\begin{lem}\label{isotopyconnected} Let $O$ be an isotopically SLH connected open subset of a space $X$. For any two points $x$ and $y$ from $O$ there exists a homeomorphism $h\colon X\to X$, supported on $O$ and isotopic to the identity via an isotopy supported on $O$, such that $h(x) = y$.
\end{lem}

\begin{proof} Let $O'$ be the set of all points $z\in O$ with the following property: there exists a homeomorphism $h\colon X\to X$, supported on $O$ and isotopic to the identity via an isotopy supported on $O$, such that $h(x) = z.$ A standard argument can be used to show that $O'$ is both open and closed in $O.$ Hence $O' =O$. This implies the statement of the lemma.
\end{proof}

Since any manifold is isotopically SLH, we obtain the following corollary.

\begin{cor}\label{supported} Let $X$ be a manifold and $U$ be a connected open subset of $X$. For any $x$ and $y$ from $U$ there exists a homeomorphism $h\colon X\to X$ supported on $U$ such that $h(x) = y$, and such that $h$ is isotopic to the identity via an isotopy supported on $U.$
\end{cor}

By a {\it  path} joining two distinct points $a$ and $b$ in a space $X$ we mean an image of $[0,1]$ in $X$ under a homeomorphism $\gamma \colon [0,1]\to X$ such that $\gamma(0)=a$ and $\gamma (1)=b.$
The proof of the next fact is straightforward.

\begin{lem}\label{isotopySLH} Let $X$ be an isotopically SLH space. For any path $P$ joining two distinct points $x$ and $y$ of $X$ and any open neighbourhood $U$ of $P$ there exists a homeomorphism $h\colon X\to X$ such that $h(x)=y$ and such that $h$ is isotopic to the identity via an isotopy supported on $U.$
\end{lem}

Let $Y$ be a subspace of $X$. We say that a space $X$ is \textit{isotopically strongly $n$-homogeneous relative to $Y$} if any bijections between two $n$-element subsets of $X\setminus Y$ can be extended to an autohomeomorphism of $X$ which is isotopic to the identity via an isotopy supported on $X\setminus Y$. If, in the above definition, $Y=\varnothing$ then $X$ is called  \textit{isotopically strongly $n$-homogeneous}.

\begin{thm}\label{isotopicallyhomogeneous} Let $Y$ be a closed subspace of a space $X$. If  $X\setminus Y$ is isotopically SLH space such that no finite subset separates it, then $X$ is isotopically strongly $n$-homogeneous relative to $Y$ for any $n.$
\end{thm}

\begin{proof}  Let $A=\{a_k\}_{k=1}^n$ and $B=\{b_k\}_{k=1}^n$ be two $n$-element subsets of $X\setminus Y$. Since $X\setminus Y$ is isotopically SLH, we may assume that $A\cap B=\varnothing.$ For each $k=1,2,\dots, n$ let $O_k = \{a_k\}\cup \{b_k\} \cup (X\setminus (A\cup B\cup Y)).$ By Lemma~\ref{isotopyconnected}, for each $k$ there exists a homeomorphism $f_k\colon X\to X$, supported on $O_k$ and isotopic to the identity via an isotopy supported on $O_k$, such that $f_k(a_k) = b_k$. The homeomorphism $f= f_n f_{n-1}\dots f_1$ is isotopic to the identity via an isotopy supported on $X\setminus Y$, and such that $f(a_k) = b_k$ for all $k.$	
\end{proof}

\begin{cor}\label{isotopicallyhomogeneousmanifold} Any connected manifold (resp. any connected manifold with boundary) is isotopically strongly $n$-homogeneous for any $n$ (resp. isotopically strongly $n$-homogeneous relative to the boundary for any $n$).
\end{cor}

The following  fact is obvious.

\begin{lem}\label{piecewisedefined} Let $\mathcal U =\{U_{\alpha}\}_{\alpha\in \mathcal A}$ be a discrete collection of open subsets of a space $X$ and let $f_{\alpha}$ be homeomorphisms supported on $U_{\alpha}$, $	\alpha\in \mathcal A$. Then there exists a homeomorphism $f\colon X\to X$ such that $f|U_{\alpha} = f_\alpha.$ Moreover, if each of $f_{\alpha}$ is isotopic to the identity via an isotopy supported on $U_{\alpha}$, then $f$ is isotopic to the identity via an isotopy supported on $\bigcup\{ U_\alpha \colon \alpha \in \mathcal A\}.$
\end{lem}

%\begin{lem}\label{discisotopy} Let $B^n$ denote the closed $n$-ball, $S^n$ be the corresponding $n$-sphere, and $x\in B^n\setminus S^n$. Let also $U$, $V$, and $W$ be neighbourhoods of $x$ in $B^n$ such that $U\subset \overline{U}\subset V\subset \overline{V}\subset W$ and $\overline{W}\cap S^n = \varnothing.$ Then there exists a homeomorphism $h\colon B^n\to B^n$ supported on $W$  such that $V\subset f(U)$ and such that $h$ is isotopic to the identity  via an isotopy supported on $W$.
%\end{lem}

%\begin{lem}\label{productisotopy} Consider $p\in (0,1)$ and let three positive numbers $\varepsilon$ , $\delta_1$ and $\delta_2$ be such that   $\delta_1<\delta_2$ and $\delta_2+\varepsilon < \min\{p, 1-p\}.$ For any space $Z$  there exists a homeomorphism $h\colon Y\times [0,1] \to Y\times [0,1],$ supported on $Y\times (\delta_1, 1-\delta_1)$, such that

%\begin{itemize}
	%\item $Y\times (\delta_2, 1-\delta_2) \subset h(Y\times (\varepsilon, 1-\varepsilon))$;
	%\item $h$ is isotopic to the identity  via an isotopy  $H$ supported on $Y\times (\delta_1, 1-\delta_1);$
	%\item $h$ depends only on the second coordinate of $Y\times [0,1]$, and $H$ depends only on the second coordinate and $t$.

%\end{itemize}
 %\end{lem}

\section{Discrete homogeneity}

In this section, we state and prove our results.

% In what follows, by a {\it  path} joining two points $a$ and $b$ in a space $X$ we either mean an image of $[0,1]$ in $X$ under a homeomorphism $\gamma \colon [0,1]\to X$ such that $\gamma(0)=a$ and $\gamma (1)=b$ if $a\ne b$, or a singleton $\{a\}$ if $a=b$.

\begin{lem}\label{discretesubsets} Let $X$ be a generalized continuum and $E=\{Y_n\}_{n=1}^\infty$ be an end of $X$. Then there exists a countable discrete subset $A$ of $X$ such that $\overline{A}\cap E(X)= \{E\}$, where $\overline{A}$ denotes the closure of $A$ in $F(X).$
	\end{lem}
\begin{proof} For each $n=1,2,\dots$ pick a point $a_n\in Y_n$ and let $A=\{ a_n\}_{n=1}^\infty.$ It is easy to verify that $A$ satisfies the conclusion of the lemma.
	\end{proof}

\begin{pro}\label{oneend}
If, for a  generalized continuum $X$, $|E(X)| >1$ then $X$ is not discrete homogeneous.	
\end{pro}
\begin{proof}
Consider $a,b\in E(X)$ such that $a\ne b$. By Lemma~\ref{discretesubsets}, there exist two discrete countable subsets $A$ and $B$ of $X$ such that $\overline{A}\cap E(X) =\{ a\}$ and $\overline{B}\cap E(X) =\{ b\}$ (here, for a subset $Y$ of $X$, by $\overline{Y}$ we denote the closure of $Y$ in $F(X)$). Suppose there is a homeomorphism $f\colon X\to X$ such that $f(A\cup B) = A$. Then, for a homeomorphism $\overline{f}\colon F(X)\to F(X)$ extending $f$ we have $\overline{f}(a)=a=\overline{f}(b)$, contradiction.
\end{proof}

For a metric space $X$, $x\in X$ and $r\ge 0$ we let $B(x,r)= \{y\in X\colon {\rm dist} (x,y) \le r\}.$ For a collection $\mathcal S$ of sets by $\bigcup \mathcal S$ we denote the union of all sets in $\mathcal S$.

\begin{thm}\label{manifold_3} An $m$-dimensional connected  manifold $M$, $m>2$, is strongly discrete homogeneous if and only if it has at most one end.
	\end{thm}

\begin{proof}
	The ``only if" part follows from Proposition~\ref{oneend}. For the other direction, let $M$ be an $m>2$ manifold with $|E(M)|\le 1$. If $M$ has no ends, then by Remark~\ref{remark_components}(c) $M$ is compact, and hence by Remark~\ref{slh} strongly discrete homogeneous. Assume now that $M$ has exactly one end.  Fix a metric on $M$ in which every closed bounded set is compact \cite{V}. Fix a point $O\in M$ and let $B(r)=B(O,r).$  Let  $A$ be a discrete countable subsets of $M$ so that  $A = \{a_k\colon k \geq 1\}$. There exists a sequence of positive numbers $s_1 < s_2 < \dots < s_k < \dots$ such that $\displaystyle \lim_{k\to\infty} s_k =\infty.$ and $A\cap \mathrm{Bd}\, B(s_k)=\varnothing$ for all $k.$ Since $M$ is connected, $\mathrm{Bd}\, B(s_k)\ne \varnothing$, and we can pick a point $b_k\in \mathrm{Bd}\, B(s_k)$, $k=1,2,\dots$ Clearly, the set $B=\{b_k\colon k \geq 1\}$ is discrete. Moreover, $A\cap B=\varnothing,$ and $A\cup B$ is discrete. It is easy to see that to prove the theorem it is sufficient to  construct an autohomeomorphism of $M$  that extends a bijection sending $a_k$ to $b_k$, $k =1,2, \dots.$

	 For  each $k$ let
	$\mathcal P_k$ be a collection of all  paths joining $a_k$ and $b_k$ in $M$ that avoid all other $a$'s and $b$'s. Let
	$$r_k = \sup\{ r \ge 0 \colon \exists P\in\mathcal P_k \mbox{ such that } P\cap B(r) = \varnothing \}.$$

 Since $M$ has only one end, for any $r>0$ the subspace  $M\setminus B(r)$ has only one unbounded component
 (see Remark~\ref{remark_components}(c)).  Further, local connectedness of $M$ implies  the existence of $r'>r$ such that all bounded components of $M\setminus B(r)$ are contained in $B(r')$ (see Remark~\ref{remark_components}(a)). Since $A\cup B$ is discrete, the set $B(r')$ contains only finitely many points from $A\cup B.$ Therefore there exists $k'$ such that all pairs $a_k$, $b_k$ with $k\ge k'$ are in the unbounded component of $M\setminus B(r).$ It follows that
 $\displaystyle \lim _{k\to \infty} r_k = \infty.$

Using Lemma~\ref{enlargement}, for each $k$ choose open connected neighbourhoods $U_k$ and $V_k$ of $a_k$ and $b_k$, respectively, such that $U_k\cap (A\cup B) = \{a_k\}$, $V_k\cap (A\cup B) = \{b_k\},$ and the family $\{ U_k, V_k\}_{k=1}^\infty$ is discrete.
Let  $0< \epsilon_k <1$, $k=1,2,\dots,$ be such that $B(a_k,\varepsilon_k)\subset U_k$ and $B(b_k,\varepsilon_k)\subset V_k.$ 	For each $k$ let $P_k$ be a path joining $a_k$ and $b_k$, that avoids $B(r_k/2).$ Note that this conditions implies that each path $P_i$ can intersect at most finitely many paths $P_j$. Therefore we can use Lemma~\ref{generalposition} to approximate each path $P_k$ by another path $P'_k$ with the endpoints $a'_k\in B(a_k,\varepsilon_k)$ and $b'_k\in B(b_k,\varepsilon_k)$, such that $P'_i\cap P'_j = \varnothing$ for $i\ne j$. The condition $\displaystyle \lim _{k\to \infty} r_k = \infty$ implies that the family $\{P'_k\}_{k=1}^\infty$ is discrete. By Corollary~\ref{supported}, for each $k$ there exists a homeomorphism $f_k\colon M\to M$ that maps $a_k$ to $a'_k$ and is supported on $U_k$. Similarly, for each $k$ there exists a homeomorphism $g_k\colon M\to M$ that maps $b_k$ to $b'_k$ and is supported on $V_k$. By Lemma~\ref{piecewisedefined}, the collections of homeomorphisms $\{f_k\}_{k=1}^\infty$ and $\{g_k\}_{k=1}^\infty$ define homeomorphisms $f\colon M\to M$ and $g\colon M\to M$ such that $f(a_k)=a'_k$ and $g(b_k)=b'_k$ for all $k.$ Further, appllying Lemma~\ref{enlargement}, find a discrete family $\{ O_k \}_{k=1}^\infty$ of open connected neighbourhoods of paths $P'_k.$ Applying Corollary~\ref{supported} again, for each $k$ find a homeomorphism $h_k\colon M\to M$ such that $h_k(a'_k) = b'_k$ and which is supported on  $O_k$. By Lemma~\ref{piecewisedefined},  the collection of homeomorphisms $\{h_k\}_{k=1}^\infty$ defines a homeomorphisms $h\colon M\to M$ such that $h(a'_k) = b'_k$ for all $k.$ Finally, the homeomorphism $g^{-1}hf$ sends $a_k$ to $b_k$ for each $k$, as required.
	\end{proof}

The following is a generalization of Proposition~2.14 from \cite{ChK}.

\begin{lem}\label{planewithholes} Let $\mathcal D$ be a discrete family of closed disks in $\mathbb R^2$. Let $A$ and $B$ be two subsets of $\mathbb R^2\setminus \bigcup \mathcal D$ such that $A\cup B$ is discrete in $\mathbb R^2$. For any bijection between $A$ and $B$ there exists a homeomorphism $f \colon \mathbb R^2 \to \mathbb R^2$ extending this bijection and supported on $\mathbb R^2\setminus \bigcup \mathcal D.$ Moreover, we can choose $f$ so that it is isotopic to the identity via an isotopy supported on $\mathbb R^2\setminus \bigcup \mathcal D.$
\end{lem}

\begin{proof} For each $D\in \mathcal D$, let $c_D$ denote its centre and $r_D$ denote its radius. Let $C=\{ c_D\colon D\in \mathcal D.\}$ Further, let $A=\{a_k\}_{k=1}^\infty$ and $B=\{b_k\}_{k=1}^\infty$ be two infinite discrete  subsets of $\mathbb R^2$ such that $A\cup B\subset  \mathbb R^2\setminus \bigcup \mathcal D.$ We want to extend a bijection between $A$ and $B$ that sends $a_k$ to $b_k$, $k=1,2,\dots,$ to an autohomeomorphism of $\mathbb R^2$, satisfying conditions of the lemma.

Let $E=A\cup B = \{e_k\}_{k=1}^\infty$. For each $k$ and $D\in\mathcal D$ there exist $\varepsilon_k>0$ and $\gamma_D>0$ such that $\gamma_D > r_D$ and $\{ B(e_k, \varepsilon _k)\}_{k=1}^\infty \cup \{ B(c_D,\gamma_D)\}_{D\in\mathcal D}$ is a discrete family of subsets of $\mathbb R^2$ (see Lemma~\ref{enlargement}). For each pair $i$ and $j$ such that $a_i=b_j$ there exists a unique $k=k(i,j)$ such that $a_i=e_k=b_j$. Further, for each such $k$ there exists a homeomorphism $\phi_k\colon \mathbb R^2 \to \mathbb R^2,$ supported on $\mathrm{Int}\, B(e_k, \varepsilon_k),$  and  isotopic to the identity via an isotopy supported on $\mathrm{Int}\, B(e_k, \varepsilon_k),$ such that $\phi_k (a_i)\ne b_j.$    Applying Lemma~\ref{piecewisedefined}, we obtain a homeomorphism $\phi \colon \mathbb R^2\to \mathbb R^2,$ supported on $\cup_{k=1}^\infty \mathrm{Int}\, B(e_k, \varepsilon_k),$ and isotopic to the identity via an isotopy supported on $\cup_{k=1}^\infty \mathrm{Int}\, B(e_k, \varepsilon_k),$ such that $\phi (A) \cap B =\varnothing.$ Replacing $A$ and $B$ by $\phi(A)$ and $B$, we will assume that $A\cap B=\varnothing.$

Since there are only countably many of perpendicular bisectors of segments joining pairs of points from $A\cup B \cup C$, there exists a point $O\in \mathbb R^2$ such that all distances from $O$ to the points in $A\cup B \cup C$ are distinct and positive. 

Find a ray $R$ in $\mathbb R^2$ emanating from $O$ such that $R$ contains no points from $A\cup B\cup C$. Let also $\{R_k\}_{k=1}^\infty $ be a collection of distinct rays in $\mathbb R^2$ emanating from $O$ and converging to $R$,  and such that $R_k$ contains no points from $A\cup B\cup C$ for all $k.$

 Denote by $S_r$  a circle of radius $r$ centred at $O$,  by $S_r(\epsilon)$ an $\epsilon$-neighbourhood of $S_r$, and by $\overline{S}_r(\epsilon)$ the closure of $S_r(\epsilon)$. Further, for each $p\in \mathbb R^2$ let $\|p\|=\mathrm{dist}\, (O,p).$

For each $k$ there exists $\alpha_k>0$  and $\beta_k>0$ such that:

\begin{itemize}

\item $\overline{S}_{\|a_k\|} (\alpha_k)\cap C =\varnothing =\overline{S}_{\|b_k\|} (\beta_k)\cap C$;

\item $\overline{S}_{\|a_i\|} (\alpha_i)\cap \overline{S}_{\|a_j\|} (\alpha_j) =\varnothing$ for all $i\ne j;$

\item $\overline{S}_{\|b_i\|} (\beta_i)\cap \overline{S}_{\|b_j\|} (\beta_j) =\varnothing$ for all $i\ne j;$

\item $\overline{S}_{\|a_i\|} (\alpha_i)\cap \overline{S}_{\|b_j\|} (\beta_j) =\varnothing$ for all $i,j$.

\end{itemize}

For each $D\in \mathcal D$ there exists $\delta_D>0$ such that $\delta_D< r_D$ and
$$B(c_D,\delta_D)\cap ((\cup_{k=1}^\infty  S_{\|a_k\|} (\alpha_k))\cup (\cup_{k=1}^\infty  S_{\|b_k\|} (\beta_k))\cup (\cup_{k=1}^\infty R_k))=\varnothing.$$

  For each $k$, there exist autohomeomorphisms $h^A_k$ and $h^B_k$ of $\mathbb R^2$, supported on $S_{\|a_k\|} (\alpha_k)$ and $S_{\|b_k\|} (\beta_k),$  respectively, such that

\begin{itemize}

 \item $h^A_k (a_k) \in R_k$ and $\| h^A_k(a_k) \| = \| a_k\|$;
  	
 \item $h^A_k (b_k) \in R_k$ and $\| h^B_k(b_k) \| = \| b_k\|$;

%\item $h^A_k$ is isotopic to the identity via an isotopy supported on $S_{\|a_k\|} (\alpha_k);$

%\item  $h^B_k$ is isotopic to the identity via an isotopy supported on $S_{\|b_k\|} (\beta_k).$
  	
 \end{itemize}

  Applying Lemma~\ref{piecewisedefined}, we obtain a homeomorphism $h\colon \mathbb R^2\to R^2$ supported on $S=(\cup_{k=1}^\infty  S_{\|a_k\|} (\alpha_k))\cup (\cup_{k=1}^\infty  S_{\|b_k\|} (\beta_k))$ such that %and isotopic to the identity via an isotopy supported on $S$, such that

\begin{itemize}

\item  $h(a_k)\in R_k$ and $\| a_k\| =\| h(a_k)\|;$

\item $h(b_k)\in R_k$ and $\| b_k\| =\| h(b_k)\|.$

\end{itemize}

For each $k=1,2,\dots$ let $U_k(\epsilon)$ denote an open $\epsilon$-neighbourhood of the segment joining $h(a_k)$ and $h(b_k).$ For each $k$ there exists $\rho_k>0$ such that $ \{B(c_D,\delta_D)\}_{D\in\mathcal D}\cup \{ U_k(\rho_k)\}_{k=1}^\infty$ is a discrete family of subsets of $\mathbb R^2.$ By Lemma~\ref{isotopySLH}, for each $k$ there exists a homeomorphism $g_k\colon \mathbb R^2\to\mathbb R^2,$ supported on $U_k(\rho_k)$, and isotopic to the identity via an isotopy supported on $U_k(\rho_k)$, such that $g_k(h(a_k))= h(b_k).$ By Lemma~\ref{piecewisedefined}, there exists a homeomorphism $g\colon \mathbb R^2\to \mathbb R^2$, supported on $U=\cup_{k=1}^\infty U_k (\rho_k)$, and isotopic to the identity on $U$, such that $g(h(a_k)) = h(b_k)$ for all $k.$

There exists a homeomorphism  $\psi \colon \mathbb R^2\to \mathbb R^2$  that maps each disk $D\in \mathcal D$ onto $B(c_D,\delta_D),$  such that

\begin{itemize}
		
	\item $\psi(c_D)= c_D$ for all $D\in\mathcal D$;
	
	\item $\psi$ is supported on $\cup\{ \mathrm{Int}\, B(C_D, \gamma_D)\colon D\in\mathcal D\}.$
	
\end{itemize}

Let $f=\psi^{-1}  h^{-1} gh\psi.$ It is easy to check that $f$ is the homeomorphism we are looking for.
\end{proof}

Following the ideas of \cite{P}, we say that a space $X$ is \textit{ isotopically strongly discrete homogeneous} if every bijection between  discrete subsets of $X$  can be
extended to an autohomeomorphism of $X$ which is isotopic to the identity.
The next theorem gives a positive answer to a question from \cite{P}.

\begin{thm} A $2$-dimensional connected  manifold $M$ is isotopically strongly discrete homogeneous if and only if it has at most one end.
\end{thm}

\begin{proof}
	Since being isotopically strongly discrete homogeneous implies sDH, the ``only if" part follows from Proposition~\ref{oneend}.
	
	 If $M$ has no ends, then it is compact and therefore any discrete subset of $M$ is finite. Hence the statement of the theorem follows from Corollary~\ref{isotopicallyhomogeneousmanifold}. 
	
	 Assume now that $M$ is a $2$-manifold with exactly one end. It follows from Theorem~3 of \cite{R} that $M$ can be obtained from $\mathbb R^2$ as follows. Let $\mathcal D^1_1$, $\mathcal D^2_1$, and $\mathcal D_2$ be families of disks in $\mathbb R^2$ such that $\mathcal D = \mathcal D^1_1 \cup \mathcal D^2_1 \cup \mathcal D_2$ is a discrete (and hence at most countable) family. Additionally,  $\mathcal D^i_1 =\{ D^i_{1,j} \colon j\in J\},$ $i=1,2,$ and $D_2=\{ D_{2,l} \colon l\in L\}$, where each of $J$, $L$ is either finite or countable.  Remove interiors of disks from $\mathcal D$ and glue a handle (i.e. $\mathbb S^1\times [0,1]$) $H_j$ to each pair of boundaries of disks $D^1_{1,j}$ and $D^2_{1,j}$, $j\in J$, and a M\"obius band $M_l$ to  the boundary  of every disk $D_{2,l}$, $l\in L.$ The resulting manifold is $M$.

	For each $\epsilon >0$ and $D\in\mathcal D$ let $D(\epsilon)$ denote the closed $\epsilon$-neighbourhood of the boundary of $D$ in $\mathbb R^2\setminus\cup \{ \mathrm{Int}\, D\colon D\in\mathcal D\}$. Further, for each $\epsilon >0$ and for each $j\in J$, $l\in L$ let $H_j(\epsilon) = H_j\cup D^1_{1,j} (\epsilon)\cup D^2_{1,j}(\epsilon),$ and $M_l(\epsilon) = M_l \cup D_{2,l} (\epsilon),$ respectively. Note that $H_j(\epsilon)$ is homeomorphic to $H_j$ and $M_l(\epsilon)$ is homeomorphic to $M_l$ for all $j\in J$, $l\in L$. 
	
	For each $j\in J$, $l\in L$ there exist $\epsilon_j$ and $\epsilon_l$ such that:
	
	\begin{itemize}
		\item the family $\{ D^1_{1,j}(\epsilon_j), D^2_{1,j}(\epsilon_j)\colon j\in J\} \cup \{ D_{2,l} (\epsilon_l)\colon l\in L\}$ is discrete in $M$;
		\item for all $j$, $l,$ the interiors (in $\mathbb R^2$) $\mathrm{Int}\, D^1_{1,j}(\epsilon_j)$, $\mathrm{Int}\, D^2_{1,j}(\epsilon_j)$, $\mathrm{Int}\, D_{2,l} (\epsilon_l)$ do not contain points from $A\cup B$.		
	\end{itemize}

Let $A$ and $B$ be two discrete countable subsets of $M.$ Applying Corollary~\ref{isotopicallyhomogeneousmanifold} to each manifold with boundary $H_j(\epsilon_j)$, $M_l(\epsilon_l)$, for each $j$ and $l$ we obtain homeomorphisms $f_j\colon H_j(\epsilon_j)\to H_j(\epsilon_j)$ and $f_l\colon M_l(\epsilon_l)\to M_l(\epsilon_l)$, such that

\begin{itemize}
	\item $f_j$ is isotopic to the identity via an isotopy supported on $H_j(\epsilon_j)\setminus (\mathrm{Bd}\, D^1_{1,j}(\epsilon_j)\cup \mathrm{Bd}\, D^2_{1,j}(\epsilon_j)) $;

	\item $f_l$ is  isotopic to the identity via an isotopy supported on $M_l(\epsilon_l)\setminus (\mathrm{Bd}\, D_{2,l}(\epsilon_l))$;
	
	\item $f_j((A\cup B)\cap H_j)\subset \mathrm{Int}\, D^1_{1,j}(\epsilon_j)$;

	\item $f_l((A\cup B)\cap M_l)\subset \mathrm{Int}\, D_{2,l}(\epsilon_l)$.

\end{itemize}

Next, using Lemma~\ref{piecewisedefined}, we obtain a homeomorphism $f\colon M\to M$, which is isotopic to the identity and such that $f(A\cup B)\subset \mathbb R^2 \setminus \bigcup \mathcal D.$  Now the statement of the theorem follows from Lemma~\ref{planewithholes}.
	\end{proof}

\begin{cor} A $2$-dimensional connected  manifold $M$ is strongly discrete homogeneous if and only if it has at most one end.
\end{cor}

\begin{prob} Let $X$ be a metrizable strongly locally homogeneous generalized continuum with one end. Is $X$  strongly discrete homogeneous or discrete homogeneous?
	\end{prob}

\section{Acknowledgements}

After this paper has been posted on arXiv.org, it was brought to our attention that the result stated in Theorem~\ref{manifold_3} has been obtained in~\cite{P}. We are grateful to Riccardo Piergallini who pointed this out. We also are grateful to the anonymous referee for valuable suggestions that significantly improved the manuscript.

%\newpage

\vskip1cm

\noindent(V.A. Chatyrko)\\
Department of Mathematics, Linkoping University, 581 83 Linkoping, Sweden.\\
vitalij.tjatyrko@liu.se

\vskip0.3cm
\noindent(A. Karassev) \\
Department of Mathematics,
Nipissing University, North Bay, Canada\\
alexandk@nipissingu.ca

\end{document}